\documentclass{article}
\usepackage{amsmath,amsthm,amsfonts,amssymb}
\usepackage[pdftex,pdfborder={0 0 0}]{hyperref}
\usepackage[margin=1in]{geometry}
\usepackage{multirow}
\usepackage{array}
\usepackage{bbm}
\usepackage[noblocks]{authblk}
\usepackage{verbatim}
\usepackage{tikz}
\usepackage{float}
\usepackage{enumerate}
\usepackage{diagbox}
\usepackage{bm}
\usepackage{url}
\newtheorem{theorem}{Theorem}[section]
\newtheorem{lemma}[theorem]{Lemma}
\newtheorem{proposition}[theorem]{Proposition}
\newtheorem{corollary}[theorem]{Corollary}

\theoremstyle{definition}
\newtheorem{definition}[theorem]{Definition}

\newtheorem{problem}[theorem]{Open Problem}

\newcommand{\N}{\mathbb{N}}
\renewcommand{\mod}[2]{\equiv#1\textup{ (mod }#2\textup{)}}

\title{On the Domination Number of Permutation Graphs and an Application to Strong Fixed Points}

\author[1]{Theresa Baren}
\author[2]{Michael Cory}
\author[2]{Mia Friedberg}
\author[3]{Peter Gardner}
\author[4]{James Hammer}
\author[4]{Joshua Harrington}
\author[5]{Daniel McGinnis}
\author[6]{Riley Waechter}
\author[7]{Tony W. H. Wong}

\affil[1]{Auburn University}
\affil[2]{University of Florida}
\affil[3]{Western Carolina University}
\affil[4]{Cedar Crest College}
\affil[5]{New College of Florida}
\affil[6]{Northern Arizona University}
\affil[7]{Kutztown University of Pennsylvania}

\begin{document}
\maketitle

\begin{abstract}
A \textit{permutation graph} $G_\pi$ is a simple graph with vertices corresponding to the elements of $\pi$ and an edge between $i$ and $j$ when $i$ and $j$ are inverted in $\pi$. A set of vertices $D$ is said to $dominate$ a graph $G$ when every vertex in $G$ is either an element of $D$, or adjacent to an element of $D$. The \textit{domination number} $\gamma(G)$ is defined as the cardinality of a minimum dominating set of $G$. A \textit{strong fixed point} of a permutation $\pi$ of order $n$ is an element $k$ such that $\pi^{-1}(j)<\pi^{-1}(k)$ for all $j<k$, and $\pi^{-1}(i)>\pi^{-1}(k)$ for all $i>k$.
In this article, we count the number of connected permutation graphs on $n$ vertices with domination number $1$ and domination number $\frac{n}{2}$. We further show that for a natural number $k\leq \frac{n}{2}$, there exists a connected permutation graph on $n$ vertices with domination number $k$. We find a closed expression for the number of permutation graphs dominated by a set with two elements, and we find a closed expression for the number of permutation graphs efficiently dominated by any set of vertices. We conclude by providing an application of these results to strong fixed points, proving some conjectures posed on the OEIS.
\end{abstract}

\section{Introduction}

Let $G$ be a simple graph. Let $V(G)$ and $E(G)$ be the vertex set and the edge set of $G$, respectively. Two vertices $u,v\in V(G)$ are called \textit{adjacent} if  $\{u,v\}\in E(G)$. For each vertex $v\in V(G)$, the \textit{open neighborhood} $N(v)$ is defined to be the set of vertices in $G$ adjacent to $v$, while the \textit{closed neighborhood} is defined to be $N[v] = N(v)\cup\{v\}$. If $U \subseteq V(G)$, let $N[U] = \bigcup_{u\in U} N[u]$. A \textit{dominating set} of the graph $G$ is a set $D\subseteq V(G)$ such that $N[D] = V(G)$. It follows that a \textit{minimum dominating set} is a dominating set $M$ such that $|M|=\min\{|D|\ :\ D\text{ is a dominating set of }G\}$.
We define the \textit{domination number} $\gamma(G)$ as the cardinality of a minimum dominating set.
An \textit{efficient dominating set} is a dominating set $D$ such that for all distinct $d_1,d_2\in D$, N[$d_1]\cap N[d_2]=\emptyset$. If $D$ is a dominating set of some graph $G$ and for $v \in V(G)$, $N[v] \cap D= \{w\}$, then $v$ is a \textit{private neighbor} of $w$ with respect to $D$. If for $u\in V(G)$ $|N[u] \cap D| \geq 2$, then $u$ is a \textit{shared neighbor} with respect to $D$. When the dominating set $D$ is clear from context, we will not use the phrase ``with respect to $D$''.

Let $S_n$ be the set of all permutations on $[n]=\{1,2,\dotsc,n\}$ for $n\in\N$. We will always reference permutations with one-line notation, e.g. if $\pi$ is the permutation $\bigl(\begin{smallmatrix}
    1 & 2 & 3 & 4 & 5 \\
    3 & 1 & 2 & 5 & 4
  \end{smallmatrix}\bigr)$, then we will write $\pi = [3,1,2,5,4]$. If the image of an element under a permutation is not explicit, we write an underscore in its place. The \textit{permutation graph} of $\pi\in S_n$, denoted by $G_\pi$, is the graph with vertex set $V(G_\pi)=\{1,2,\dotsc,n\}$ and the edge set $E(G_\pi)=\{\{i,j\}:i,j\in V(G_\pi),i<j,\text{ and }\pi^{-1}(i)>\pi^{-1}(j)\}$.

Permutation graphs were introduced by Even, Pnueli, and Lempel in \cite{EvenPermutation} and \cite{EvenPermutation1972}. A survey of properties of permutation graphs and related algorithms can be found in \cite{GolumbicAlgorithmic}. A couple of these properties are that every induced subgraph of a permutation graph is a permutation graph and the complement of a permutation graph is a permutation graph. Domination in graphs has many variants and has been well-studied; a survey of results on domination can be found in \cite{HaynesFundamentals}. 
Many domination problems such as the minimum cardinality dominating set problem and the weighted independent dominating set problem (\cite{ChaoOptimal}, \cite{BrandstDomination}) are NP-complete for general graphs but can be solved in polynomial time for permutation graphs. Algorithms and complexity on permutation graphs have been studied for this reason.

In this paper, we investigate the structural properties of permutation graphs with a given domination number. We give a recursive formula for the number of permutation graphs on $n$ vertices with at least one minimum dominating set of size one and the number of permutation graphs on $n$ vertices with exactly $t$ dominating sets of size one. We fully characterize the connected permutation graphs with domination number $n/2$ for even $n$, and we show that there exists a connected permutation graph on $n$ vertices with domination number $k$ for $1\leq k \leq n/2$. We count the number of permutation graphs on $n$ vertices dominated by a set $\{u,v\}$ and the number of permutation graphs on $n$ vertices efficiently dominated by a set of vertices $A$. We give a formula for the number of disconnected permutation graphs on $n$ vertices with domination number $k$ in terms of connected permutation graphs of smaller order. We also provide an algorithm to find a dominating set of a given permutation graph. Finally, we apply our results for permutation graphs to count the number of permutations with a given number of strong fixed points, proving some conjectures given on the On-Line Encyclopedia of Integer Sequences (OEIS), \cite{OEIS}.

\section{Permutation Graphs with Domination Number $1$}
We begin with the following definitions.
\begin{definition}
Let $g(n,k)$ denote the number of permutation graphs on $n$ vertices with domination number $k$.
\end{definition}

\begin{definition}
Let $f(n,k,t)$ denote the number of permutation graphs on $n$ vertices with domination number $k$ and precisely $t$ minimum dominating sets.
\end{definition}
In this section we will be primarily concerned with the values of $g(n,k)$ and $f(n,k,t)$ when $k=1$, and the following lemma will be useful to study this case.

\begin{lemma}\label{lem1}
Let $G_{\pi}$ be a permutation graph on $n$ vertices; $\{k\}$ is a dominating set of $G_{\pi}$ if and only if $\pi^{-1}(k)=(n+1)-k$, $\pi^{-1}(j) > \pi^{-1}(k)$ for all $j<k$, and $\pi^{-1}(i) < \pi^{-1}(k)$ for all $i > k$.
\end{lemma}
\begin{proof}
If $\{k\}$ is a dominating set of $G_{\pi}$, we know that for all $j < k$, $\pi^{-1}(j) > \pi^{-1}(k)$, and for all $i>k$, $\pi^{-1}(i) < \pi^{-1}(k)$. Therefore $\pi^{-1}(k) \leq n-(k-1)$ and $\pi^{-1}(k) \geq (n-k)+1$, so $\pi^{-1}(k)=(n+1)-k$.

For the other direction, if $\pi^{-1}(i) < \pi^{-1}(k)$ for all $i > k$ and $\pi^{-1}(j) > \pi^{-1}(k)$ for all $j<k$, then clearly $\{k\}$ is a dominating set.
\end{proof}

\begin{proposition}\label{prop1}
There are precisely $(n-k)!(k-1)!$ permutation graphs on $n$ vertices that have $\{k\}$ as a dominating set.
\end{proposition}
\begin{proof}
For any given permutation $\pi$ whose permutation graph $G_{\pi}$ has $\{k\}$ as a dominating set, by Lemma \ref{lem1}, $\{\pi^{-1}(j)\ :\ j <k\}=\{n-k+2, \dots,n\}$ and $\{\pi^{-1}(i)\ :\ i>k\}=\{1,\dots,n-k\}$. Therefore, there are $(k-1)!$ ways that $\pi^{-1}$ can act on the $k-1$ elements less than $k$, and there are $(n-k)!$ ways that $\pi^{-1}$ can act on the $n-k$ elements greater than $k$. Additionally by Lemma \ref{lem1}, there is only one way that $\pi^{-1}$ can act on $k$, thus the number of permutation graphs on $n$ vertices that have $\{k\}$ as a dominating set is $(n-k)!(k-1)!$.
\end{proof}

The following lemma will allow us to prove the recursive formula for $g(n,1)$ given in Theorem \ref{dom1count}.
\begin{lemma}\label{lem2}
There are $(n-k)!f(k-1,1,0)$ permutation graphs on $n$ vertices that have $\{k\}$, but not $\{r\}$, as a dominating set for $1 \leq r \leq k-1$.
\end{lemma}
\begin{proof}
We first show that the number of ways that the $k-1$ elements less than $k$ can be arranged in a permutation that is dominated by $\{k\}$, but not $\{r\}$, for $1 \leq r \leq k-1$ is $f(k-1,1,0)$. Consider the set $A=\{\pi \in S_n~:\ G_{\pi} \text{ is dominated by } \{k\}, \text{ but not } \{r\}, \text{ for } 1 \leq r \leq k-1 \}$. By Lemma \ref{lem1}, $\{\pi^{-1}(j)\ :\ j <k\}=\{n-k+2, \dots,n\}$, so denoting $R=\{n-k+2, \dots,n\}$, the cardinality of the set $B=\{\pi_{|R}\ :\ \pi \in A\}$ is the number of ways that the $k-1$ elements less than $k$ can be arranged. We will show that $B$ is in bijection with the set $C=\{\pi \in S_{k-1}\ :\ \gamma(G_{\pi}) \neq 1\}$.

Define the map $h:\ B \rightarrow C$ by $h(\sigma)=\tau$ where $\tau(1)=\sigma(n-k+2)$, $\tau(2)=\sigma(n-k+3),\dots ,\tau(k-1)=\sigma(n)$. The permutation $\tau$ is an element of $C$, since by definition of $B$, $\sigma=\pi_{|R}$ for some $\pi \in A$, and by definition of $A$, for all $i\geq k>r$, $\pi^{-1}(i) \leq \pi^{-1}(k) < \pi^{-1}(r)$. Also by definition of $A$, either $\pi^{-1}(j) > \pi^{-1}(r)$ for some $r < j \leq k-1$ or $\pi^{-1}(j) < \pi^{-1}(r)$ for some $j < r$ for all $1 \leq r \leq k-1$. By Lemma \ref{lem1}, this proves that $\tau$ does not have a dominating set of size one. Moreover, $h$ is clearly well-defined. To show $h$ is injective, if $\sigma \neq \pi$, then $\sigma(n-k+1+i) \neq \pi(n-k+1+j)$ for some $i \neq j$, then $h(\sigma)(i) \neq h(\pi)(j)$. To show $h$ is surjective, let $\tau \in C$, define a permutation $\pi \in S_n$ by $\pi(i)=(n+1)-i$ for $1 \leq i \leq k$ and $\pi(n-k+1+j)=\tau(j)$ for $1 \leq j \leq k-1$. Clearly $\pi \in A$ so $\pi_{|R} \in B$, and by construction, $h(\pi_{|R})=\tau$. This shows that $h$ is a bijection. Therefore, the number of ways that the $k-1$ elements less than $k$ can be arranged is $|B|=|C|=(k-1)!-g(k-1,1)=f(k-1,1,0)$. Akin to the proof from in Proposition \ref{prop1}, there are $(n-k)!$ ways that the $n-k$ elements greater than $k$ can be arranged. Consequently, the number of permutation graphs on $n$ vertices that have $\{k\}$, but not $\{r\}$, as a dominating set for $1 \leq r \leq k-1$  is $(n-k)!f(k-1,1,0)$.
\end{proof}

\begin{theorem}\label{dom1count}
A recursive formula for $g(n,1)$ is given by the following.
\begin{align}
g(0,1)&=0,\\
g(n,1)&= \sum_{k=1}^n(n-k)!f(k-1,1,0),\  \text{for } n \geq 1.\label{eq2}
\end{align}\\[-36pt]

\end{theorem}
\begin{proof}
Let $D$ denote the set of all permutation graphs on $n$ vertices that have domination number 1. Let $D_1$ be the set of all permutation graphs on $n$ vertices that have $\{1\}$ as a dominating set, and denote $D_k$ as the set of all permutation graphs on $n$ vertices that have $\{k\}$, but not $\{r\}$, as a dominating set for $1 \leq r \leq k-1$ and $2 \leq k \leq n$. Clearly each $D_i$ is disjoint and $\bigcup_{i=1}^n D_i=D$, so $|D|=\sum_{i=1}^n|D_i|=\sum_{k=1}^n(n-k)!f(k-1,1,0)$, where the last equality follows from Lemma \ref{lem2}.
\end{proof}
Although Equation \ref{eq2} does not initially seem recursive, we have the equality $f(k-1,1,0)=(k-1)!-g(k-1,1)$, so $g(n,1)$ is calculated recursively. We now give a recursive formula for $f(n,1,t)$ using similar methods.

\begin{lemma}\label{lem3}
There are $f(n-k,1,t-1)f(k-1,1,0)$ permutation graphs on $n$ vertices that have $\{k\}$, but not $\{r\}$, as a dominating set for $1 \leq r \leq k-1$, and exactly $t$ dominating sets of size one.
\end{lemma}
\begin{proof}
Let $A$ denote the set of permutations whose permutation graph has $\{k\}$, but not $\{r\}$, as a dominating set for $1 \leq r \leq k-1$, and exactly $t$ dominating sets of size one. As in Lemma \ref{lem2}, the number of ways that the $k-1$ elements less than $k$ can be arranged is $f(k-1,1,0)$. We show that the number of ways that the $n-k$ elements greater than $k$ can be arranged is $f(n-k,1,t-1)$.

Let $R=\{1,2,\dots,n-k\}$. The cardinality of the set $B=\{\pi_{|R}\ :\ \pi \in A \}$ is the number of ways that the $n-k$ elements greater than $k$ can be arranged. Let $C$ be the set of permutations on $n-k$ that has exactly $t-1$ dominating sets of size one. Analogous to the proof of Lemma \ref{lem2}, the function $h:\ B \rightarrow C$ defined by $h(\sigma)=\tau$ where $\tau(k)=\sigma(k)-k$ is a bijection, which demonstrates that the number of ways that the $n-k$ elements greater than $k$ can be arranged is $f(n-k,1,t-1)$.

Accordingly, there are $f(n-k,1,t-1)f(k-1,1,0)$ permutation graphs on $n$ vertices that have $\{k\}$, but not $\{r\}$, as a dominating set for $1 \leq r \leq k-1$, and exactly $t$ dominating sets of size one.
\end{proof}
\begin{theorem}\label{recthm}
A recursive formula for $f(n,1,t)$ is given by the following.
\begin{align}
f(n,1,0)&=n!-g(n,1),\\
f(n,1,t)&=\sum_{k=1}^{n-t+1}f(n-k,1,t-1)f(k-1,1,0) \text{ for } 1 \leq t \leq n.
\end{align}
\end{theorem}
\begin{proof}
Denote $D$ as the set of permutation graphs that have exactly $t$ dominating sets of size one.
Let $D_1$ be the set of permutation graphs on $n$ vertices that have $\{1\}$ as a dominating set and exactly $t$ dominating sets of size one. Let $D_k$ be the set of permutation graphs on $n$ vertices that have $\{k\}$, but not $\{r\}$, as a dominating set for $1 \leq r \leq k-1$ and have exactly $t$ dominating sets of size one for $2 \leq k \leq n-t+1$ (notice there are no permutation graphs on $n$ vertices that have $\{k\}$, but not $\{r\}$, as a dominating set for $1 \leq r \leq k-1$ and have exactly $t$ dominating sets for $k \geq n-t+2$). Clearly each $D_i$ is disjoint and $\bigcup_{i=1}^{n-t+1}D_i=D$, therefore $|D|=\sum_{i=1}^{n-t+1}|D_i|=\sum_{i=1}^{n-t+1}f(n-k,1,t-1)f(k-1,1,0)$, where the last equality follows from Lemma \ref{lem3}.
\end{proof}

\section{Permutation Graphs with a Dominating Set of Size Two and Permutation Graphs Efficiently Dominated by a Set of Vertices $A$}
In Proposition \ref{prop1}, we found the number of permutation graphs on $n$ vertices that have $\{k\}$ as a dominating set. We extend this result by finding the number of permutation graphs on $n$ vertices that have $\{u,v\}$ as a dominating set. We then find the number of permutation graphs on $n$ vertices that are efficiently dominated by a set of vertices $A$. The primary technique used throughout this section is analyzing the private and shared neighbors of each element in the dominating set and determining where they must lie in the one-line notation of the permutation.
\begin{lemma}\label{lem4}
The permutation graph $G_{\pi}$ has $\{u,v\}$ as a dominating set where $u < v$, and $u$ and $v$ are not adjacent if and only if every element is either $u$, $v$, a private neighbor of $u$, a private neighbor of $v$, or a shared neighbor of $u$ and $v$, and
\begin{enumerate}
	\item The private neighbors of $u$ are less than $v$,
    \item The private neighbors of $v$ are greater than $u$,
    \item The shared neighbors of $u$ and $v$ are either less than $u$ or greater than $v$.
\end{enumerate}
\end{lemma}
\begin{proof}
Assume that $G_{\pi}$ has $\{u,v\}$ as a dominating set where $u$ and $v$ are not adjacent. It then follows by definition that every element is either $u$, $v$, a private neighbor of $u$, a private neighbor of $v$, or a shared neighbor of $u$ and $v$. We prove statements $1$, $2$, and $3$ by contraposition. Note that since $u$ and $v$ are not adjacent, $\pi^{-1}(u) < \pi^{-1}(v)$.

If some private neighbor $i$ of $u$ is greater than $v$, then $i>u$, so since $i$ is not adjacent to $v$, $\pi^{-1}(v) < \pi^{-1}(i)<\pi^{-1}(u)$, which implies that $u$ and $v$ are adjacent and proves the contrapositive of statement $1$.
The proof for statement $2$ is similar.
If $u<j<v$ is a shared neighbor of $u$ and $v$, then $\pi^{-1}(j)<\pi^{-1}(u)$ and $\pi^{-1}(v)<\pi^{-1}(j)$, so $\pi^{-1}(v)<\pi^{-1}(u)$, which shows that $u$ and $v$ are adjacent and confirms the contrapositive of statement $3$.

The other direction of the statement follows directly.
\end{proof}


\begin{theorem}\label{thm1}
The number of permutation graphs on $n$ vertices that are dominated by $\{u,v\}$, where $u$ and $v$ are not adjacent and $u < v$, is given by the following expression:
\[
\sum_{\substack{x_1+x_2=u-1\\y_1+y_2=v-u-1\\z_1+z_2=n-v}}(y_1+z_2)!(x_1+z_1)!(x_2+y_2)!{u-1 \choose x_1}{v-u-1 \choose y_1}{n-v \choose z_1}
\]
for nonnegative integers $x_1$, $x_2$, $y_1$, $y_2$, $z_1$, and $z_2$.
\end{theorem}
\begin{proof}
Let $\pi$ be a permutation on $[n]$ where $u$ and $v$ are not adjacent. Let $A^{\pi}$ be the set of private neighbors of $u$ not including $u$, let $B^{\pi}$ be the set of private neighbors of $v$ not including $v$, and let $C^{\pi}$ be the shared neighbors of $u$ and $v$. By Lemma \ref{lem4}, we can conclude that $\{u,v\}$ dominates $G_{\pi}$ if and only if $A^{\pi}$, $B^{\pi}$, and $C^{\pi}$ are the disjoint unions 
\begin{align*}
A^{\pi}&=A_1^{\pi} \cup A_2^{\pi},\\
B^{\pi}&=B_1^{\pi} \cup B_2^{\pi},\\
C^{\pi}&=C_1^{\pi} \cup C_2^{\pi},
\end{align*}
where 
\begin{align*}
A_1^{\pi}&=\{x \in A^{\pi}\ :\ x<u\},& A_2^{\pi}&=\{x \in A^{\pi}\ :\ u<x<v\},\\
B_1^{\pi}&=\{x \in B^{\pi}\ :\ u<x<v\},&
B_2^{\pi}&=\{x \in B^{\pi}\ :\ x>v\},\\
C_1^{\pi}&=\{x \in C^{\pi}\ :\ x<u\},&
C_2^{\pi}&=\{x \in C^{\pi}\ :\ x>v\},
\end{align*}
and 
\begin{align*}
|A_1^{\pi}|+|C_1^{\pi}|&=u-1,\\
|A_2^{\pi}|+|B_1^{\pi}|&=v-u-1,\\
|B_2^{\pi}|+|C_2^{\pi}|&=n-v.\\
\end{align*}
Now denote $D{(x_1,x_2,y_1,y_2,z_1,z_2)}$ where $x_1$, $x_2$, $y_1$, $y_2$, $z_1$, and $z_2$ are natural numbers and $x_1+x_2=u-1$, $y_1+y_2=v-u-1$, and $z_1+z_2=n-v$ as the set of permutations $\pi$ such that $|A_1^{\pi}|=x_1$, $|C_1^{\pi}|=x_2$, $|A_2^{\pi}|=y_1$, $|B_1^{\pi}|=y_2$, $|B_2^{\pi}|=z_1$, and $|C_2^{\pi}|=z_2$. A permutation $\tau$ is an element of $D{(x_1,x_2,y_1,y_2,z_1,z_2)}$ if and only if
\[
\tau=[ \underline{\ \ } , \dots , \underline{\ \ } , u, \underline{\ \ } , \dots , \underline{\ \ } , v, \underline{\ \ } , \dots ,\underline{\ \ } ],
\]
where every element in $A^{\tau}_2$ and $C^{\tau}_2$ lies to the left of $u$ in any arrangement, every element in $A^{\tau}_1$ and $B^{\tau}_2$ lies in between $u$ and $v$ in any arrangement, and every element in $B^{\tau}_1$ and $C^{\tau}_1$ lies to the right of $v$ in any arrangement. There are also ${u-1 \choose x_1}$ ways to choose $x_1$ elements to be in $A^{\tau}_1$ (the elements in $C^{\tau}_1$ are then determined), ${v-u-1 \choose y_1}$ ways to choose elements in $A^{\tau}_2$, and ${n-v \choose z_1}$ ways to choose elements in $B^{\tau}_2$. Thus, $|D{(x_1,x_2,y_1,y_2,z_1,z_2)}|=(y_1+z_2)!(x_1+z_1)!(x_2+y_2)!{u-1 \choose x_1}{v-u-1 \choose y_1}{n-v \choose z_1}$. Now let $D$ be the set of permutations whose graphs are dominated by $u$ and $v$, where $u$ and $v$ are not adjacent, then $D=\bigcup D{(x_1,x_2,y_1,y_2,z_1,z_2)}$ where the union is taken over all tuples $(x_1,x_2,y_1,y_2,z_1,z_2)$ such that $x_1+x_2=u-1$, $y_1+y_2=v-u-1$, and $z_1+z_2=n-v$. Since this is also a disjoint union, 
\begin{align*}
|D|&=\sum|D{(x_1,x_2,y_1,y_2,z_1,z_2)}|\\
&=\sum_{\substack{x_1+x_2=u-1\\y_1+y_2=v-u-1\\z_1+z_2=n-v}}(y_1+z_2)!(x_1+z_1)!(x_2+y_2)!{u-1 \choose x_1}{v-u-1 \choose y_1}{n-v \choose z_1}.
\end{align*}

\end{proof}
The following lemma is proved in a similar fashion to Lemma \ref{lem5}, so we omit the proof.
\begin{lemma}\label{lem5}
The permutation graph $G_{\pi}$ has $\{u,v\}$ as a dominating set where $u < v$, and $u$ and $v$ are adjacent if and only if every element is a private neighbor of $u$, a private neighbor of $v$, or a shared neighbor of $u$ and $v$, and 
\begin{enumerate}
	\item The private neighbors of $u$ are greater than $u$,
    \item The private neighbors of $v$ are less than $v$.
\end{enumerate}
\end{lemma}

Notice that there are no conditions on the shared neighbors of $u$ and $v$.
\begin{theorem}\label{thm2}
The number of permutation graphs on $n$ vertices that are dominated by $\{u,v\}$ where $u$ and $v$ are adjacent and $u < v$ is
\[
\sum_{\substack{x_1+x_2+x_3=v-u-1\\y_1+y_2=u-1\\z_1+z_2=n-v}}(x_1+z_2)!(z_1+x_3+y_1)!(y_2+x_2)!{v-u-1 \choose x_1}{v-u-1-x_1\choose x_2}{u-1 \choose y_1}{n-v \choose z_1}
\]
for nonnegative integers $x_1$, $x_2$, $x_3$, $y_1$, $y_2$, $z_1$, and $z_2$.
\end{theorem}
\begin{proof}
Let $D$ be the set of all permutations on $[n]$ whose  graphs are dominated by $\{u,v\}$ where $u$ and $v$ are adjacent. Let $\pi$ be a permutation of $[n]$ where $u$ and $v$ are adjacent in $G_{\pi}$. Let $A^{\pi}$ be the set of private neighbors of $u$, $B^{\pi}$ the private neighbors of $b$, and $C^{\pi}$ the shared neighbors of $u$ and $v$ other than $u$ and $v$. By Lemma \ref{lem5}, $A^{\pi}$, $B^{\pi}$, and $C^{\pi}$ are the disjoint unions
\begin{align*}
A^{\pi}&=A_1^{\pi} \cup A_2^{\pi},\\
B^{\pi}&=B_1^{\pi} \cup B_2^{\pi},\\
C^{\pi}&=C_1^{\pi} \cup C_2^{\pi} \cup C_3^{\pi},
\end{align*}
where
\begin{align*}
A_1^{\pi}&=\{x \in A^{\pi}\ :\ u<x<v\},&
A_2^{\pi}&=\{x \in A^{\pi}\ :\ x>v\},\\
B_1^{\pi}&=\{x \in B^{\pi}\ :\ x<u\},&
B_2^{\pi}&=\{x \in B^{\pi}\ :\ u<x<v\},\\
\end{align*}\\[-45pt]
\[
C_1^{\pi}=\{x \in C^{\pi}\ :\ x<u\},\
C_2^{\pi}=\{x \in C^{\pi}\ :\ u<x<v\},\
C_3^{\pi}=\{x \in C^{\pi}\ :\ x>v\},\
\]
and
\begin{align*}
&|A_1^{\pi}|+|B_2^{\pi}|+|C_2^{\pi}|=v-u-1,\\
&|B_1^{\pi}|+|C_1^{\pi}|=u-1,\\
&|A_2^{\pi}|+|C_3^{\pi}|=n-v.
\end{align*}
Now denote $D{(x_1,x_2,x_3,y_1,y_2,z_1,z_2)}$ where $x_1$, $x_2$, $x_3$, $y_1$, $y_2$, $z_1$, and $z_2$ are natural numbers and $x_1+x_2+x_3=v-u-1$, $y_1+y_2=u-1$, and $z_1+z_2=n-v$ as the set of permutations $\pi$ on $[n]$ such that $|A_1^{\pi}|=x_1$, $|B_2^{\pi}|=x_2$, $|C_2^{\pi}|=x_3$, $|B_1^{\pi}|=y_1$, $|C_1^{\pi}|=y_2$, $|A_2^{\pi}|=z_1$, and $|C_3^{\pi}|=z_2$. Then a permutation, $\tau$, is an element of $D{(x_1,x_2,x_3,y_1,y_2,z_1,z_2)}$ if and only if 
\[
\tau=~[~\underline{\ \ },\dots,\underline{\ \ },v,\underline{\ \ },\dots,\underline{\ \ },u,\underline{\ \ },\dots,\underline{\ \ }~]
\]
where every element in $A_1^{\pi}$ and $C_3^{\pi}$ lies to the left of $v$ in the one-line notation of $\pi$ in any arrangement, every element in $A_2^{\pi}$, $B_1^{\pi}$, and $C_2^{\pi}$ lies in between $v$ and $u$ in any arrangement, and every element in $C_1^{\pi}$ and $B_2^{\pi}$ lies to the right of $u$ in any arrangement. Since there are ${v-u-1 \choose x_1}{v-u-1-x_1\choose x_2}$ ways to choose the elements in $A_1^{\pi}, B_2^{\pi}$, and $C_2^{\pi}$; ${u-1 \choose y_1}$ ways to choose elements in $B_1^{\pi}$ and $C_1^{\pi}$; and ${n-v \choose z_1}$ ways to choose elements in $A_2^{\pi}$ and $C_3^{\pi}$, 
\[
|D{(x_1,x_2,x_3,y_1,y_2,z_1,z_2)}|=(x_1+z_2)!(z_1+x_3+y_1)!(y_2+x_2)!{v-u-1 \choose x_1}{v-u-1-x_1\choose x_2}{u-1 \choose y_1}{n-v \choose z_1}.
\]
Now let $D$ be the set of all permutations on $[n]$ whose permutation graph is dominated by $\{u,v\}$ where $u$ and $v$ are adjacent. Then $D$ is the disjoint union 
\[
D=\bigcup D{(x_1,x_2,x_3,y_1,y_2,z_1,z_2)}
\]
where the union is taken over all tuples such that $x_1+x_2+x_3=v-u-1$, $y_1+y_2=u-1$, and $z_1+z_2=n-v$. Therefore,
\begin{align*}
|D|&= \sum |D{(x_1,x_2,x_3,y_1,y_2,z_1,z_2)}| \\
&=\sum_{\substack{x_1+x_2+x_3=v-u-1\\y_1+y_2=u-1\\z_1+z_2=n-v}}(x_1+z_2)!(z_1+x_3+y_1)!(y_2+x_2)!{v-u-1 \choose x_1}{v-u-1-x_1\choose x_2}{u-1 \choose y_1}{n-v \choose z_1}.
\end{align*}
\end{proof}
We can now deduce the following, extending Proposition $\ref{prop1}$ to dominating sets of size two.
\begin{corollary}
The number of permutation graphs on $n$ vertices that are dominated by $\{u,v\}$  is given by the sum of the expressions in Theorem \ref{thm1} and Theorem \ref{thm2}.
\end{corollary}

We now consider counts for the number of permutation graphs on $n$ vertices that are efficiently dominated by a set $A$.

\begin{lemma}\label{lem6}
Let $G_{\pi}$ be a permutation graph on $n$ vertices, and let $A=\{a_1,\dots,a_k\}\subseteq V(G_{\pi})$ where $a_1<\cdots <a_k$. The set $A$ is an efficient dominating set of $G_{\pi}$ if and only if $A$ is a dominating set of $G_{\pi}$, no two elements of $A$ are adjacent, no two elements of $A$ have shared neighbors, and the private neighbors of $A$ satisfy the following conditions.
\begin{enumerate}
	\item The private neighbors of $a_1$ are less than $a_2$, and the private neighbors of $a_k$ are greater than $a_{k-1}$,
    \item For $2 \leq i \leq k-1$, the private neighbors of $a_i$ are greater than $a_{i-1}$ and less than $a_{i+1}$.
\end{enumerate}
\end{lemma}
\begin{proof}
Assume $A$ is an efficient dominating set of $G_{\pi}$, then $A$ is a dominating set, no two elements of $A$ are adjacent, and no two elements of $A$ have shared neighbors. Since no two elements of $A$ are adjacent, for $1 \leq i \leq n-1$, $a_i$ lies to the left of $a_{i+1}$ in the one-line notation of $\pi$. To prove statement $1$, assume $a_1$ has a private neighbor $u$ greater than $a_2$. Then $u$ lies to the left of $a_1$ and hence to the left of $a_2$ contradicting that no two elements of $A$ have shared neighbors. A similar argument applies for the second statement and the second half of the first statement.

The other direction of the lemma follows immediately.
\end{proof}

\begin{theorem}
The number of permutation graphs on $n$ vertices that are efficiently dominated by $A=\{a_1,\dots,a_k\}$ where $a_1<\cdots <a_k$ is
\[
\sum_{x_{1}+x_{2}=a_2-a_1-1}x_{1}!((a_1-1)+(n-a_2))!x_{2}!{a_{2}-a_1-1 \choose x_{1}}
\]
if $k=2$,
\[
\sum_{\substack{x_{1,1}+x_{1,2}=a_2-a_1-1\\x_{2,1}+x_{2,2}=a_3-a_2-1}}x_{1,1}!((a_1-1)+x_{2,1})!((n-a_3)+x_{1,2})!x_{2,2}!{a_{2}-a_1-1 \choose x_{1,1}}{a_{3}-a_2-1 \choose x_{2,1}}
\]
if $k=3$, and
\[
\sum_{\substack{x_{1,1}+x_{1,2}=a_2-a_1-1\\x_{2,1}+x_{2,2}=a_3-a_2-1\\ \vdots \\ x_{k-1,1}+x_{k-1,2}=a_k-a_{k-1}-1}}x_{1,1}!((a_1-1)+x_{2,1})!x_{k-1,2}!((n-a_k)+x_{k-2,2})!\prod_{i=2}^{k-2}(x_{i,2}+x_{i+2,1})!\prod_{j=1}^{k-1}{a_{j+1}-a_j-1 \choose x_{j,1}}
\]
if $k \geq 4$
where all $x_{i,j}$'s are natural numbers.
\end{theorem}
\begin{proof}
Let $\pi$ be a permutation. For all $1 \leq i \leq k$, denote $A_i^{\pi}$ as the private neighbors of $a_i$, not including $a_i$. Let $A_{i,1}^\pi$ and $A_{i,2}^\pi$ be as follows:
	\begin{center}
		\begin{tabular}{c|c|c}
			$i$&$A_{i,1}^\pi$&$A_{i,2}^\pi$ \\
            \hline
            1 & $\{x \in A_1^{\pi}\ :\ x<a_1\}$ & $\{x \in A_1^{\pi}\ :\ a_1<x<a_2\}$ \\
            $2\leq i \leq k-1$ & $\{x \in A_i^{\pi}\ :\ a_{i-1}<x<a_{i}\}$ & $\{x \in A_i^{\pi}\ :\ a_i<x<a_{i+1}\}$\\
            $k$ & $\{x \in A_k^{\pi}\ :\ a_{k-1}<x<a_k\}$ & $\{x \in A_k^{\pi}\ :\ a_k<x\}$.
		\end{tabular}
	\end{center}
    	By Lemma \ref{lem6}, $A$ is an efficient dominating set of $G_{\pi}$ if and only if
\begin{enumerate}
	\item $a_i$ lies to the left of $a_{i+1}$ in the one-line notation of $\pi$ for $1 \leq i \leq k-1$,
    \item each $A_i^{\pi}$ is the disjoint union
\begin{align*}
A_i^{\pi}&=A_{i,1}^{\pi} \cup A_{i,2}^{\pi},
\end{align*}
and for $1 \leq i \leq k-1$,
\begin{align*}
&|A_{i,2}^{\pi}|+|A_{i+1,1}^{\pi}|=a_{i+1}-a_i-1,\\
&|A_{1,1}^{\pi}|=a_1-1,\\
&|A_{k,2}^{\pi}|=n-a_k.
\end{align*}
\end{enumerate}
Now let $x=(x_{1,1},x_{1,2},\dots,x_{k-1,1},x_{k-1,2})$ be a tuple of natural numbers such that $x_{i,1}+x_{i,2}=a_{i+1}-a_i-1$ for all $1 \leq i \leq k-1$, and denote $D(x)$ as the set of all permutations $\tau$ whose permutation graphs are efficiently dominated by $A$ and for $1 \leq i \leq k-1$,
\begin{align*}
&|A_{i+1,1}^{\tau}|=x_{i,2},\\
&|A_{i,2}^{\tau}|=x_{i,1},\\
&|A_{1,1}^{\tau}|=a_1-1,\\
&|A_{k,2}^{\tau}|=n-a_k.
\end{align*}
Now we have that the elements lying between $a_i$ and $a_{i+1}$ in the one-line notation of $\tau$ are precisely the elements in $A_{i,1}^{\tau}$ and $A_{i+1,2}^{\tau}$, which can be in any arrangement between $a_i$ and $a_{i+1}$. We also have that precisely the elements in $A_{1,2}^{\tau}$ lie to the left of $a_1$ in any arrangement, and precisely the elements in $A_{k,1}^{\tau}$ lie to the right of $a_k$ in any arrangement. Also there are ${a_{i+1}-a_i-1 \choose x_{i,1}}$ ways to choose the $x_{i,1}$ elements for $A_{i,2}^{\tau}$ and the elements for $A_{i+1,1}^{\tau}$ are then determined. Thus, we have that 
\[
|D|=x_{1,1}!((a_1-1)+x_{2,1})!x_{k-1,2}!((n-a_k)+x_{k-2,2})!\prod_{i=2}^{k-2}(x_{i,2}+x_{i+2,1})!\prod_{j=1}^{k-1}{a_{j+1}-a_j-1 \choose x_{j,1}}
\]
when $k \geq 4$,
\[
x_{1,1}!((a_1-1)+x_{2,1})!((n-a_3)+x_{1,2})!x_{2,2}!{a_{2}-a_1-1 \choose x_{1,1}}{a_{3}-a_2-1 \choose x_{2,1}}
\]
when $k=3$, and 
\[
x_{1}!((a_1-1)+(n-a_2))!x_{2}!{a_{2}-a_1-1 \choose x_{1}}
\]
when $k=2$.
Now if we let $D$ be the set of all permutations on $[n]$ whose permutation graphs are efficiently dominated by $A$, then $D$ is the disjoint union 
\[
D=\bigcup_t D(t) 
\]
where the union is taken over all tuples $t$ that satisfy the same conditions as $x$ above. Therefore,
\[
|D|=\sum_t |D(t)|,
\] which gives us the desired result.
\end{proof}
\section{Connected Permutation Graphs with Domination Number $n/2$}
In Section 2, we found the exact number of connected permutation graphs on $n$ vertices that have domination number $1$. It is a well-known result that if a connected graph on $n$ vertices has domination number $k$, then $1 \leq k \leq \lfloor n/2 \rfloor$, (\cite{HaynesFundamentals}). In this section, we find the exact number of connected permutation graphs on $n$ vertices with the upper extreme domination number value, $n/2$, for even $n$.
\begin{lemma}\label{degree}
Let $1\leq i\leq n$, and let $\pi\in S_n$. If the degree of $i$ in $G_\pi$ is $k$, then $|\pi^{-1}(i)-i|\leq k$ and $\pi^{-1}(i)-i\mod{k}{2}$. In particular, if $i$ is a leaf in $G_\pi$, then $\pi^{-1}(i)\in\{i-1,i+1\}$.
\end{lemma}

\begin{proof}
Say $\pi^{-1}(i)-i = j \geq 0$ (the same arguments hold when $j<0$). Then in the one-line notation of $\pi$, $i$ is in the $(i+j)^{\rm th}$ position. There are $n-i$ elements greater than $i$, but only $n-i-j$ positions available to the right of $i$, so at least $j$ will come before $i$, and hence the degree of $i$ is at least $j$. Equality holds if and only if all elements to the right of $i$ are greater than $i$.\\
To show that the parities are the same, we look at the number $\ell$ of elements smaller than $i$ to the right of $i$. Each of these elements forces an element larger than $i$ to come before $i$, so the degree of $i$ is in fact $j+2\ell$, which will always have the same parity as $j$.
\end{proof}

\begin{theorem}[\cite{Fink1985}]\label{halfdomination}
Let $n\geq6$ be an even integer. A connected graph $G$ with $n$ vertices has domination number $n/2$ if and only if the vertex set $V(G)$ can be partitioned into two subsets of size $n/2$, $V_1$ and $V_2$, such that the induced subgraph on $V_1$ is a connected graph, the induced subgraph on $V_2$ is an empty graph, and the edges between $V_1$ and $V_2$ form a perfect matching.
\end{theorem}
The following lemma gives us a simple criterion that determines if a permutation graph is disconnected.
\begin{lemma}\label{dislemma1}
A permutation graph $G_{\pi}$ on $n$ vertices is disconnected if and only if for some $1 \leq k \leq n-1$, $\{\pi(1),\dots,\pi(k)\}=\{1,\dots,k\}$.
\end{lemma}
\begin{proof}
If for some $1 \leq k \leq n-1$, $\{\pi(1),\dots,\pi(k)\}=\{1,\dots,k\}$, then no vertex in $\{1,\dots,k\}$ is adjacent to a vertex in $\{k+1,\dots,n\}$. This implies $G_{\pi}$ is disconnected.

Next, assume that $G_{\pi}$ is disconnected. Let $H$ be the connected component that contains the vertex $1$, and let $s$ be the smallest vertex such that $s\notin V(H)$. Let $V_1=\{i\in V(H):i<s\}$ and $V_2=\{i\in V(H):i>s\}$. If $V_2\neq\emptyset$, then since $s$ is not connected to any vertex of $H$, for all $i\in V_1$, $\pi^{-1}(i)<\pi^{-1}(s)$; and for all $i\in V_2$, $\pi^{-1}(i)>\pi^{-1}(s)$. This implies $V_1$ and $V_2$ are disconnected, a contradiction. Therefore, $V_2=\emptyset$, and by letting $k=s-1$, we have $V(H)=\{1,2,\dotsc,k\}$. Moreover, for all $i\leq k$ and for all $j>k$, since $i$ is not connected to $j$, we have $\pi^{-1}(i)<\pi^{-1}(j)$. As a result, $\{\pi(1),\dots,\pi(k)\}=\{1,\dots,k\}$.
\end{proof}


\begin{definition} Let $n$ be an even positive integer. A graph $G$ with $n$ vertices is a \textit{comb} if the vertex set $V(G)$ can be partitioned into two subsets of size $n/2$, $V_1$ and $V_2$, such that the induced subgraph on $V_1$ is a path, the induced subgraph on $V_2$ is an empty graph, and the edges between $V_1$ and $V_2$ form a perfect matching.
%
%
%
%
%
\end{definition}

\begin{theorem}\label{halfdominationpermutation}
Let $n\geq6$ be an even integer. A connected permutation graph $G_{\pi}$ on $n$ vertices has domination number $n/2$ if and only if $G_{\pi}$ is a comb. Furthermore, there are exactly $2$ such permutation graphs on $n$ vertices with domination number $n/2$.
\end{theorem}

\begin{proof}
First, we show that a comb is a permutation graph. Let $\sigma,\tau\in S_n$ be permutations defined in the following way.\\
\textit{Case $1$:} $n\mod{0}{4}$.
$$\sigma(i)=\left\{\begin{array}{ll}
3&\text{if }i=1,\\
n-2&\text{if }i=n,\\
i-3&\text{if }i>1\text{ and }i\mod{1}{4},\\
i-1&\text{if }i\mod{2}{4},\\
i+1&\text{if }i\mod{3}{4},\text{ and}\\
i+3&\text{if }i<n\text{ and }i\mod{0}{4}.\\
\end{array}\right.$$
$$\tau(i)=\left\{\begin{array}{ll}
1&\text{if }i=3,\\
n&\text{if }i=n-2,\\
i+1&\text{if }i\mod{1}{4},\\
i+3&\text{if }i<n-2\text{ and }i\mod{2}{4},\\
i-3&\text{if }i>3\text{ and }\mod{3}{4},\text{ and}\\
i-1&\text{if }i\mod{0}{4}.\\
\end{array}\right.$$
\textit{Case $2$:} $n\mod{2}{4}$.
$$\sigma(i)=\left\{\begin{array}{ll}
3&\text{if }i=1,\\
n&\text{if }i=n-2,\\
i-3&\text{if }i>1\text{ and }i\mod{1}{4},\\
i-1&\text{if }i\mod{2}{4},\\
i+1&\text{if }i\mod{3}{4},\text{ and}\\
i+3&\text{if }i<n-2\text{ and }i\mod{0}{4}.\\
\end{array}\right.$$
$$\tau(i)=\left\{\begin{array}{ll}
1&\text{if }i=3,\\
n-2&\text{if }i=n,\\
i+1&\text{if }i\mod{1}{4},\\
i+3&\text{if }i<n\text{ and }i\mod{2}{4},\\
i-3&\text{if }i>3\text{ and }\mod{3}{4},\text{ and}\\
i-1&\text{if }i\mod{0}{4}.\\
\end{array}\right.$$
It is easy to check that $G_\sigma$ and $G_\tau$ are combs, where the leaves in $G_\sigma$ are $i\mod{1\text{ or }0}{4}$, and those in $G_\tau$ are $i\mod{2\text{ or }3}{4}$. Together with Theorem \ref{halfdomination}, we prove that a comb is a connected permutation graph with domination number $n/2$. 

Next, let $\pi\in S_n$ be such that $G_\pi$ is a connected permutation graph with domination number $n/2$. By Theorem \ref{halfdomination}, $G_\pi$ has exactly $n/2$ leaves, and all leaves have distinct neighbors that are not leaves.\\

\noindent\underline{Claim $1$.} Exactly one of $1$ and $2$ is a leaf. Similarly, exactly one of $n-1$ and $n$ is a leaf.

\begin{proof}[Proof of Claim $1$]
If $1$ and $2$ are not leaves, then there exist leaves $i$ and $j$ adjacent to $1$ and $2$ respectively. Note that $i,j>2$, so $\pi^{-1}(i)<\pi^{-1}(1)$ and $\pi^{-1}(j)<\pi^{-1}(2)$. If $\pi^{-1}(i)<\pi^{-1}(j)$, then $i$ is adjacent to both $1$ and $2$ in $G_\pi$; if $\pi^{-1}(j)<\pi^{-1}(i)$, then $j$ is adjacent to both $1$ and $2$ in $G_\pi$. Both contradict that $i$ and $j$ are leaves. If $1$ and $2$ are both leaves, then there exist $i$ and $j$ adjacent to $1$ and $2$ respectively. With the same argument, we can deduce that $1$ and $2$ share the same neighbor, violating the structure of $G_\pi$.
\end{proof}

\noindent\underline{Claim $2$.} If there exists $2\leq i\leq n-2$ such that both $i$ and $i+1$ are leaves, then $\pi^{-1}(i)=i-1$ and $\pi^{-1}(i+1)=i+2$.

\begin{proof}[Proof of Claim $2$]
By Lemma \ref{degree}, we have four cases: (a) $\pi^{-1}(i)=i-1$ and $\pi^{-1}(i+1)=i$, (b) $\pi^{-1}(i)=i+1$ and $\pi^{-1}(i+1)=i+2$, (c) $\pi^{-1}(i)=i+1$ and $\pi^{-1}(i+1)=i$, or (d) $\pi^{-1}(i)=i-1$ and $\pi^{-1}(i+1)=i+2$.

In cases (a) and (b), let $x$ be adjacent to $i$. If $x>i$ and $\pi^{-1}(x)<\pi^{-1}(i)$, then $\pi^{-1}(x)<\pi^{-1}(i+1)$, implying that $x$ is also adjacent to $i+1$, violating the structure of $G_\pi$. If $x<i$ and $\pi^{-1}(x)>\pi^{-1}(i)$, then $\pi^{-1}(x)\geq\pi^{-1}(i)+1=\pi^{-1}(i+1)$. As $\pi^{-1}(x)\neq\pi^{-1}(i+1)$, we have $\pi^{-1}(x)>\pi^{-1}(i+1)$, which again implies that $x$ is adjacent to $i+1$. In case (c), $i$ and $i+1$ are adjacent to each other, which is impossible. Therefore, the only possibility is case (d).
\end{proof}

\noindent\underline{Claim $3$.} If there exists $1\leq i\leq n-2$ such that both $i$ and $i+2$ are leaves, then $\pi^{-1}(i)=i-1$ and $\pi^{-1}(i+2)=i+3$. In particular, this implies $1$ and $3$ cannot both be leaves.

\begin{proof}[Proof of Claim $3$]
By Lemma \ref{degree}, we have three cases: (a) $\pi^{-1}(i)=i+1$ and $\pi^{-1}(i+2)=i+3$, (b) $\pi^{-1}(i)=i-1$ and $\pi^{-1}(i+2)=i+1$, or (c) $\pi^{-1}(i)=i-1$ and $\pi^{-1}(i+2)=i+3$.

In case (a), there exists exactly one $j<i+1$ such that $\pi(j)>i$. Note that $\pi(j)<i+2$, otherwise $\pi(j)$ is adjacent to both $i$ and $i+2$. In other words, $\pi(j)=i+1$. As a result, $\{\pi(1),\pi(2),\dotsc,\pi(i+1)\}=\{1,2,\dotsc,i+1\}$, which means $G_\pi$ is disconnected by Lemma \ref{dislemma1}, giving us a contradiction. In case (b), there exists exactly one $j>i+1$ such that $\pi(j)<i+2$. Note that $\pi(j)>i$, otherwise $\pi(j)$ is adjacent to both $i$ and $i+2$. In other words, $\pi(j)=i+1$. As a result, $\{\pi(i+1),\pi(i+2),\dotsc,\pi(n)\}=\{i+1,i+2,\dotsc,n\}$, which again means $G_\pi$ is disconnected, giving us a contradiction. Therefore, the only possibility is case (c).
\end{proof}

\noindent\underline{Claim $4$.} For all $1\leq i\leq n-3$, there are at most two leaves in $\{i,i+1,i+2,i+3\}$.

\begin{proof}[Proof of Claim $4$]
If $i$, $i+1$, and $i+2$ are leaves, by Claim $2$, $\pi^{-1}(i)=i-1$ and $\pi^{-1}(i+1)=i+2$, and by Claim $2$ again, $\pi^{-1}(i+1)=i$ and $\pi^{-1}(i+2)=i+3$, giving us a contradiction. This also implies that it is impossible for all $i+1$, $i+2$, and $i+3$ to be leaves. If $i$, $i+1$, and $i+3$ are leaves, by Claim $2$, $\pi^{-1}(i)=i-1$ and $\pi^{-1}(i+1)=i+2$, and by Claim $3$, $\pi^{-1}(i+1)=i$ and $\pi^{-1}(i+3)=i+4$, giving us a contradiction. A similar argument rules out the possibility that all $i$, $i+2$, and $i+3$ are leaves.
\end{proof}

\noindent\underline{Claim $5$.} For all $1\leq i\leq n/2$, exactly one of $2i-1$ and $2i$ is a a leaf.

\begin{proof}[Proof of Claim $5$]
Assuming the contrary, let $i_1$ be the smallest integer such that there are zero or two leaves in $\{2i_1-1,2i_1\}$. By Claim $1$, there is exactly one leaf in $\{1,2\}$. This implies that there is exactly one leaf in $\{2(i_1-1)-1,2(i_1-1)\}$. By Claim $4$, there cannot be two leaves in $\{2i_1-1,2i_1\}$, so there are zero leaves in $\{2i_1-1,2i_1\}$. Similarly, let $i_2$ be the largest integer such that there are zero or two leaves in $\{2i_2-1,2i_2\}$. Since there is exactly one leaf in $\{n-1,n\}$ by Claim $1$, we can deduce that there are zero leaves in $\{2i_2-1,2i_2\}$. Moreover, by Claim $4$, whenever there is an $i_0$ such that there are two leaves in $\{2i_0-1,2i_0\}$, there will be no leaves in $\{2(i_0-1)-1,2(i_0-1)\}$ and $\{2(i_0+1)-1,2(i_0+1)\}$. As a result, the total number of leaves is strictly less than $n/2$, contradicting our assumption.
\end{proof}

\noindent\underline{Claim $6$.} If there exists $1\leq i\leq n/2-2$ such that $2i$ is a leaf, then $2i+1$ and $2i+4$ are leaves.

\begin{proof}[Proof of Claim $6$]
If $2i+2$ is a leaf, then by Claim $4$, $2i+3$ is not a leaf, and by Claim $5$, $2i+4$ is a leaf. By Claim $3$, $\pi^{-1}(2i)=2i-1$ and $\pi^{-1}(2i+2)=2i+3$, and $\pi^{-1}(2i+2)=2i+1$ and $\pi^{-1}(2i+4)=2i+5$, giving us a contradiction. Hence, if $2i$ is a leaf, then $2i+1$ is a leaf. By Claim $4$, $2i+3$ is not a leaf, and by Claim $5$, $2i+4$ is a leaf.
\end{proof}

\noindent\underline{Claim $7$.} If $1$ is a leaf, then
$$\pi(i)=\left\{\begin{array}{ll}
i-1&\text{if }i\mod{2}{4},\text{ and}\\
i+1&\text{if }i\mod{3}{4}.\\
\end{array}\right.$$
If $2$ is a leaf, then
$$\pi(i)=\left\{\begin{array}{ll}
i+1&\text{if }i\mod{1}{4},\text{ and}\\
i-1&\text{if }i\mod{0}{4}.\\
\end{array}\right.$$

\begin{proof}[Proof of Claim $7$]
If $1$ is a leaf, then $2$ and $3$ are not leaves by Claim $1$ and Claim $3$ respectively. By Claim $5$, we have that $4$ is a leaf. By applying Claim $6$ inductively, the leaves in $G_\pi$ are $i\mod{1\text{ or }0}{4}$. If $2$ is a leaf, then again by applying Claim $6$ inductively, the leaves in $G_\pi$ are $i\mod{2\text{ or }3}{4}$. The proof of this claim is completed by applying Claim $2$.
\end{proof}

\noindent\underline{Claim $8$.} Let $\sigma$ and $\tau$ be those permutations defined at the beginning of the proof of this theorem. If $1$ is a leaf, then $\pi=\sigma$; and if $2$ is a leaf, then $\pi=\tau$.

\begin{proof}[Proof of Claim $8$]
If $1$ is a leaf, then by Claim $7$, $\pi(i)=\sigma(i)$ if $i\mod{2\text{ or }3}{4}$. Hence, 
$$\tiny{\pi=[~\underline{\phantom{|1|}},\underline{\phantom{|}1\phantom{|}},\underline{\phantom{|}4\phantom{|}},\underline{\phantom{|1|}},\underline{\phantom{|1|}},\underline{\phantom{|}5\phantom{|}},\underline{\phantom{|}8\phantom{|}},\underline{\phantom{|1|}},\dotsc,\underline{\phantom{|1|}},\underline{\phantom{|}4i-3\phantom{|}},\underline{\phantom{|}4i\phantom{|}},\underline{\phantom{|1|}},\hspace{-30pt}\underset{\large\substack{\uparrow\\(4i+1)\text{-st position}}}{\underline{\phantom{|1|}}}\hspace{-30pt},\underline{\phantom{|}4i+1\phantom{|}},\underline{\phantom{|}4i+4\phantom{|}},\underline{\phantom{|1|}},\underline{\phantom{|1|}},\underline{\phantom{|}4i+5\phantom{|}},\underline{\phantom{|}4i+8\phantom{|}},\underline{\phantom{|1|}},\dotsc,\underline{\phantom{|1|}}~].}$$
Note that $\pi(1)\neq2$, else $\{\pi(1),\pi(2)\}=\{1,2\}$, which means $G_\pi$ is disconnected. Also, $\pi(1)<5$, else $\pi(1)$ is adjacent to both leaves $4$ and $5$. Therefore, $\pi(1)=3$.

For each $1\leq i\leq\lfloor n/4\rfloor$ we observe the following.
\begin{enumerate}
\item Note that $\pi^{-1}(4i-2)\neq4i$ unless $4i=n$, else $\{\pi(1),\pi(2),\dotsc,\pi(4i)\}=\{1,2,\dotsc,4i\}$, which means $G_\pi$ is disconnected. Also, $\pi^{-1}(4i-2)<4i+3$, else $4i-2$ is adjacent to both leaves $4i$ and $4i+1$. Therefore, unless $4i=n$, $\pi^{-1}(4i-2)=4i+1$, i.e., $\pi(4i+1)=4i-2$. When $4i=n$, $\pi(4i)=4i-2$. 
\item If $4i<n$, note that $\pi(4i)\neq4i+2$ unless $4i+2=n$, else $\{\pi(1),\pi(2),\dotsc,\pi(4i+2)\}=\{1,2,\dotsc,4i+2\}$, which means $G_\pi$ is disconnected. Also, $\pi(4i)<4i+5$, else $\pi(4i)$ is adjacent to both leaves $4i+4$ and $4i+5$. Therefore, unless $4i+2=n$, $\pi(4i)=4i+3$. When $4i+2=n$, $\pi(4i)=4i+2$. As a result, $\pi=\sigma$.
\end{enumerate}

If $2$ is a leaf, then by Claim $7$, $\pi(i)=\tau(i)$ if $i\mod{1\text{ or }0}{4}$. Hence,
$$\tiny{\pi=[~\underline{\phantom{|}2\phantom{|}},\underline{\phantom{|1|}},\underline{\phantom{|1|}},\underline{\phantom{|}3\phantom{|}},\underline{\phantom{|}6\phantom{|}},\underline{\phantom{|1|}},\underline{\phantom{|1|}},\underline{\phantom{|}7\phantom{|}},\dotsc,\underline{\phantom{|}4i-2\phantom{|}},\underline{\phantom{|1|}},\underline{\phantom{|1|}},\underline{\phantom{|}4i-1\phantom{|}},\hspace{-18pt}\underset{\large\substack{\uparrow\\(4i+1)\text{-st position}}}{\underline{\phantom{|}4i+2\phantom{|}}}\hspace{-18pt},\underline{\phantom{|1|}},\underline{\phantom{|1|}},\underline{\phantom{|}4i+3\phantom{|}},\underline{\phantom{|}4i+6\phantom{|}},\underline{\phantom{|1|}},\underline{\phantom{|1|}},\underline{\phantom{|}4i+7\phantom{|}},\dotsc\underline{\phantom{|1|}}~].}$$
Note that $\pi^{-1}(1)\neq2$, else $\{\pi(1),\pi(2)\}=\{1,2\}$, which means $G_\pi$ is disconnected. Also, $\pi^{-1}(1)<5$, else $1$ is adjacent to both leaves $2$ and $3$. Therefore, $\pi^{-1}(1)=3$, i.e., $\pi(3)=1$.

For each $1\leq i\leq\lfloor n/4\rfloor$ we observe the following.
\begin{enumerate}
\item Note that $\pi(4i-2)\neq4i$ unless $4i=n$, else $\{\pi(1),\pi(2),\dotsc,\pi(4i)\}=\{1,2,\dotsc,4i\}$, which means $G_\pi$ is disconnected. Also, $\pi(4i-2)<4i+3$, else $\pi(4i-2)$ is adjacent to both leaves $4i+2$ and $4i+3$. Therefore, unless $4i=n$, $\pi(4i-2)=4i+1$. When $4i=n$, $\pi(4i-2)=4i$.
\item If $4i<n$, note that $\pi^{-1}(4i)\neq4i+2$ unless $4i+2=n$, else $\{\pi(1),\pi(2),\dotsc,\pi(4i+2)\}=\{1,2,\dotsc,4i+2\}$, which means $G_\pi$ is disconnected. Also, $\pi^{-1}(4i)<4i+5$, else $4i$ is adjacent to both leaves $4i+2$ and $4i+3$. Therefore, unless $4i+2=n$, $\pi^{-1}(4i)=4i+3$, i.e., $\pi(4i+3)=4i$. When $4i+2=n$, then $\pi(4i+2)=4i$. As a result, $\pi=\tau$.
\end{enumerate}
\end{proof}

As seen in Claim $8$, $\pi$ is either equal to $\sigma$ or $\tau$, and hence, $G_\pi$ is a comb, and there are two permutation graphs with domination number $n/2$.
\end{proof}

\section{Existence of Connected Permutation Graphs with Domination Number $k$}

As we mentioned earlier, every induced subgraph of a permutation graph is a permutation graph. This implies that every connected component of a permutation graph is a permutation graph. Therefore, disconnected permutation graphs are simply the disjoint union of permutation graphs of smaller order. We are interested in connected permutation graphs because they are ``new" in the sense that they are not disjoint unions of permutation graphs of smaller order. Furthermore, since the disjoint union of permutation graphs is a permutation graph, we know there exists a permutation graph on $n$ vertices with domination number $k$ for all $1 \leq k \leq n$ by simply taking a disjoint union of $k$ permutation graphs with domination number $1$. This is another source of motivation for the study of domination numbers of connected permutation graphs. By the results from Sections $2$ and $4$, we have an exact count for the number of connected permutation graphs on $n$ vertices with domination number $1$ and an exact count for the number of connected permutation graphs with domination number $n/2$ for even $n$. We now show by inductive means that for all $n$ there exists a connected permutation graph on $n$ vertices with domination number $k$ for all $1 \leq k \leq \lfloor n/2 \rfloor$. It is well-known and shown in \cite{HaynesFundamentals} that a graph on $n$ vertices with no isolated points has domination number at most $\lfloor n/2 \rfloor$, so we show that there exists a permutation graph on $n$ vertices with domination number $k$ for each relevant value of $k$.

\begin{theorem}\label{cnctddomk}
Let $G_{\pi}$ be a connected permutation graph on $n$ vertices with domination number $k$, let $D$ be a minimum dominating set of $G_{\pi}$, and let $a \in D$ be the right-most element of $D$ in the one-line notation of $\pi$. The permutation graph $G_{\tau}$ where $\tau$ is obtained by placing $n+1$ immediately to the left of $a$ in the one-line notation is a connected permutation graph on $n+1$ vertices with domination number $k$.
\end{theorem}
\begin{proof}
If $G_{\pi}$ has domination number 1, then clearly $G_{\tau}$ has domination number 1 given by the same dominating set $D$. So assume $k \geq 2$ and assume that $G_{\tau}$ has domination number $r$, and let $S$ be a dominating set of $G_{\tau}$ of size $r$. If $r \leq k-1$ Then $S$ must contain $n+1$ otherwise $S$ would be a dominating set of $G_{\pi}$. Since $n+1$ does not dominate any elements to the left of $n+1$ in the one-line notation of $\tau$ (there are elements to the left of $n+1$ since $k \geq 2$), $S\setminus \{n+1\}$ dominates every element to the left of $n+1$. Now let $b \in D$ be the largest element of $D$. Since either $b=a$ or $b$ lies to the left of $a$ in the one-line notation of $\pi$, $b$ dominates $a$ and every element to the right of $a$. Therefore, $D'=(S\setminus \{n+1\}) \cup \{b\}$ is a dominating set of $G_{\pi}$. However since $|D'| \leq k-1$, this is a contradiction, showing that $r \geq k$. Since $D$ is a dominating set of $G_{\tau}$, $r=k$. The connectivity of $G_{\tau}$ follows from the connectivity of $G_{\pi}$ and Lemma \ref{dislemma1}, thus proving the claim.
\end{proof}

\begin{corollary}
For all $n$ there exists a connected permutation graph on $n$ vertices with domination number $k$ for all $1 \leq k \leq \lfloor n/2 \rfloor$.
\end{corollary}
\begin{proof}
We proceed by induction. Clearly the statement is true for $n=1$. Assume it is true for $n=r$. If $r+1$ is odd, then $\lfloor (r+1)/2 \rfloor = \lfloor r/2 \rfloor$, so the statement follows from Theorem \ref{cnctddomk}. If $r+1$ is even then by Theorem \ref{halfdominationpermutation} there exists a connected permutation graph with domination number $\lfloor (r+1)/2 \rfloor$. By induction and Theorem \ref{cnctddomk}, there exists a connected permutation graph on $r+1$ vertices with domination number $k$ for $1 \leq k \leq \lfloor r/2 \rfloor = \lfloor (r+1)/2 \rfloor -1$.
\end{proof}

\section{Disconnected Permutation Graphs with Domination Number $k$}

\begin{definition}
Let $c(n,k)$ denote the number of connected permutation graphs on $n$ vertices that have domination number $k$, and let $d(n,k)$ denote the number of disconnected permutation graphs on $n$ vertices that have domination number $k$.
\end{definition}

As mentioned before, each connected component of a disconnected permutation graph is a connected permutation graph of smaller order. Consequently, we count the number of disconnected permutation graphs on $n$ vertices with domination number $k$ in terms of permutation graphs of smaller order with some domination number less than $k$. First, we present a lemma that provides some information on a permutation given that its permutation graph has a connected component.

\begin{lemma}\label{dislemma}
Let $G_{\pi}$ be a permutation graph and $H \subseteq G_{\pi}$ be a subgraph with $k$ vertices whose smallest vertex is $r$, then $H$ is a connected component if and only if
\begin{enumerate}
	\item $V(H)=\{r,r+1,\dots,n+(k-1)\}$,
	\item $H$  is a connected permutation graph, and $H=G_{\tau}$, where $\tau=[\pi(r),\pi(r+1),\dots,\pi(r+(k-1))]$,
    \item $\pi^{-1}(i) < r$ for all $i<r$ and $\pi^{-1}(j)>r+(k-1)$ for all $j > r+(k-1)$,
    \item $\{\pi^{-1}(i)\ :\ i \in V(H)\}=\{r,r+1,\dots,r+(k-1)\}$.
\end{enumerate}
\end{lemma}
\begin{proof}
Let $H$ be a connected component. If $i \in V(H)$ for some $i>n+(k-1)$, then there is some $r< j \leq r+(k-1)$ such that $j \notin V(H)$. Let $V_1=\{i\in V(H)\ :\ i<j\}$, and $V_2=\{i \in V(H)\ :\ i>j\}$. Since $j \notin V(H)$, $\pi^{-1}(j)>\pi^{-1}(i)$ for all $i \in V_1$ and $\pi^{-1}(j)<\pi^{-1}(i)$ for all $i \in V_2$. However, this implies that $\pi^{-1}(r)<\pi^{-1}(s)$ for all $r \in V_1$ and $s \in V_2$, which contradicts that $H$ is connected and thus proves condition $1$.
$H$ is an induced subgraph of $G_{\pi}$, so $H$ is a connected permutation graph. Clearly there is an edge between two vertices in $H$ if and only if there is an edge between the same two vertices in $G_{\tau}$. This proves condition $2$. Notice that
$3$ follows from $1$, and $4$ follows from $3$.

The other direction of the statement follows immediately.
\end{proof}

Lemma \ref{dislemma} essentially states that the vertices of a connected component of a permutation graph $G_{\pi}$ compose a set of consecutive numbers that lie together in the one-line notation of $\pi$, and that knowledge of the smallest vertex of each component and the permutation associated to each component determines $\pi$. We utilize this concept in the following theorem.

\begin{theorem}\label{disthm}
\[
d(n,k)=\sum_{r=2}^k ~\sum_{\ell=1}^r ~\sum_{r_1+ \dots +r_{\ell}=r} ~\sum_{\substack{n_1< \cdots <n_{\ell} \\ r_1n_1+ \cdots +r_{\ell}n_{\ell}=n}} ~\sum_{\substack{k_1+ \cdots +k_{\ell}=k\\k_i \geq r_i}}\left( {r \choose r_1,r_2,\dots,r_{\ell}} \prod_{i=1}^\ell ~\sum_{k_{i_1}+ \cdots +k_{i_{r_i}}=k_i} ~\prod_{t=1}^{r_i}c(n_i,k_{i_t})\right)
\]
\end{theorem}
\begin{proof}
Let $D$ be the set of disconnected permutation graphs on $n$ vertices that have the following: 
\begin{enumerate}
\item$r$ components, 
\item$\ell$ distinctly sized components whose sizes are among and include the values $n_1,n_2, \dots, n_{\ell}$, 
\item exactly $r_i$ components of size $n_i$ for $1 \leq i \leq \ell$,
\item the property that the disjoint union of the $r_i$ components of size $n_i$ have domination number $k_i \geq r_i$ for all $i$ where $k_1+k_2+ \cdots +k_{\ell}=k$. 
\end{enumerate}
Notice that $r_1n_1+r_2n_2+ \cdots + r_{\ell}n_{\ell}=n$ and that we must require $k_i \geq r_i$ since a disjoint union of $r_i$ graphs have domination number at least $r_i$. We count the number of elements in $D$. Denote each of the $r_i$ components of size $n_i$ by  $N_{i_j}$ for $1 \leq j \leq r_i$. Each $N_{i_j}$ has some domination number $k_{i_j}$, where $k_{i_1}+k_{i_2}+ \cdots + k_{i_{r_i}}=k_i$. The number of ways that the $N_{i_j}$'s can be chosen such that the disjoint union $\cup_{j=1}^{r_i}N_{i_j}$ has domination number $k_i$ is $\prod_{t=1}^{r_i}c(n_i,k_{i_t})$ (there are $c(n_i,k_{i_1})$ choices for $N_{i_1}$, $c(n_i,k_{i_2})$ choices for $N_{i_2}$, etc.). Thus, the number of ways that $r_i$ components of size $n_i$ can have domination number $k_i$ is 
$$
\sum_{ k_{i_1}+ \cdots + k_{i_{r_i}}=k_i} \left(\prod_{t=1}^{r_i}c(n_i,k_{i_t})\right)
$$ 
for each $i$. Then by Lemma \ref{dislemma}, the permutation $\pi$ such that $G_{\pi}=\bigcup_{i=1}^\ell(\cup_{j=1}^{r_i}N_{i_j})$, where the vertices of $\bigcup_{i=1}^{\ell}(\cup_{j=1}^{r_i}N_{i_j})$ are labeled by the elements from $\{1,\dots,n\}$ is determined by knowledge of the smallest vertex of each component, $N_{i_j}$. The number of ways that the smallest vertex of each $N_{i_j}$ can be determined where the components of the same size are indistinguishable is ${r \choose r_1,r_2,\dots,r_{\ell}}$. There are $\sum_{ k_{i_1}+ \cdots + k_{i_{r_i}}=k_i} \left(\prod_{t=1}^{r_i}c(n_i,k_{i_t})\right)$ ways to choose permutations whose permutation graphs are a disjoint union of $r_i$ components of size $n_i$ that have domination number $k_i$ for each $i$. Therefore, we have that 
$$
{r \choose r_1,r_2,\dots,r_{\ell}}\prod_{i=1}^{\ell} ~\sum_{k_{i_1}+ \dots +k_{i_l}=k_i} ~\prod_{t=1}^{r_i}c(n_i,k_{i_j})
$$
is the number of permutations on $[n]$ whose permutation graphs have $r_i$ components of size $n_i$, where the disjoint union of the $r_i$ components of size $n_i$ has domination number $k_i$ for all $i$. This is the size of $D$. To get the number of disconnected permutation graphs on $n$ vertices with domination number $k$, we must then take the sum of these terms over all  tuples $(k_1,\dots,k_{\ell})$ such that $k_1+ \cdots +k_{\ell}=k$ and $k_i \geq r_i$ for all $i$. Next, take the sum over all $(r_1,\dots,r_{\ell})$ and $(n_1,\dots,n_{\ell})$ such that $r_1+ \cdots +r_{\ell}=r$, $n_1<n_2< \cdots <n_{\ell}$ and $r_1n_1+ r_2n_2+ \cdots +r_{\ell}n_{\ell}=n$ (we order the $n_i$'s this way to ensure they are distinct and avoid overcounting). We then take this sum over $\ell$ from $1$ to $r$ since there is at least $1$ and at most $r$ distinct sizes of components. Ultimately, we take the sum over all $r$ from $2$ to $k$ since a disconnected permutation graph on $n$ vertices that has domination number $k$ has at least two and at most $k$ components. This gives us our result.
\end{proof}

Notice that if we take the sum staring at $r=1$ instead of $r=2$, we get a formula for $g(n,k)$, since this would include the permutation graphs on $n$ vertices with domination number $k$. This, however, would then be given in terms of connected permutation graphs on $n$ vertices with domination number number $k$. The benefit of the the formula in Theorem \ref{disthm} is that it is in terms of connected permutation graphs with less than $n$ vertices and with domination number less than $k$.

\section{Dominating Set Algorithms}

In this section, we detail an algorithm for finding a minimum dominating set of a permutation, as well as an algorithm to find a dominating set which is often minimum. We will first walk through the procedure for finding a minimum dominating set from the adjacency matrix of a permutation graph.
\begin{enumerate}
\item Compute the adjacency matrix $A$ of the permutation graph $G_{\pi}$ of a permutation $\pi \in S_n$.
\item Add the identity matrix $I_n + A = D$. We call $D$ the domination matrix. Note that each row $i$ now represents the vertices that vertex $i$ dominates in $G_{\pi}$.
\item Let $i=1$.
\item Look for a set of $i$ rows of $D$ such that when bitwise or is applied to them, the resultant bitstring contains no zeros.
\item If no such row(s) exists, increment $i$ by $1$ and go back to step $4$.

\end{enumerate}
The minimum number of rows needed for this property to hold is the domination number, and the rows selected correspond to the vertices in a minimum dominating set. It is clear that this method will produce a dominating set. A simple contradiction argument shows that this set is minimum. \newline

We will now examine an algorithm easily done by hand to find a dominating set. Notice that decreasing subsequences of a permutation (in the one-line notation) form a clique in $G_{\pi}$, where a clique is an induced subgraph that is a complete graph. We will use this fact to create an algorithm to find a dominating set.

\begin{proposition}\label{prop2}
Let $\pi$ be a permutation on $[n]$. If $\pi$ contains $1$ adjacent to $n$ in the one-line notation ($1$ is not necessarily adjacent to $n$ in $G_{\pi}$) such that $\pi (1) \neq n$ and $\pi (n) \neq 1$, then the domination number is $2$ and a minimum dominating set is $\{ 1,n \}$.
\end{proposition}

\begin{proof}
This claim follows from the fact that any number $k$ to the left of $1$ must be larger than $1$ and so $1$ and $k$ are adjacent in $G_\pi$. A similar argument shows that every number $k$ to the right of $n$ in the permutation will be adjacent to $n$. Hence, every vertex in $G$ is dominated by either $1$ or $n$.\newline

Note that for a permutation graph to be dominated by a single vertex $v$, every number to the right of $v$ in the one-line notation of the permutation must be less than $v$ and every number to the left of $v$ must be greater than $v$. Using this, we can see it is impossible for any single vertex to dominate the graph of a permutation of the form described in this proposition, as no number in a permutation can be greater than or less than $n$ and $1$ simultaneously.
\end{proof}

\begin{proposition}\label{prop3}
Let $\pi$ be a permutation on $[n]$. If $|\pi (1) - \pi (n)|=1$, $\pi (1) \neq n$, and $\pi (n) \neq 1$,  then the domination number of $G_\pi$ is $2$ and a minimum dominating set of $G_\pi$ is $\{ \pi (1),\pi (n) \} $.
\end{proposition}
\begin{proof} 
Now, we will assume that $\pi (1) > \pi (n)$, but the same argument works for $\pi (1) < \pi (n)$. Every number $k < \pi (1)$ will be adjacent to $\pi (1)$ in $G_\pi$ and every number $j > \pi (n)$ will be adjacent to $\pi (n)$ in $G_\pi$. Because $\pi (1)$ and $\pi (n)$ are consecutive integers, this construction categorizes every element of the permutation and a dominating set of $G_\pi$ is $ \{ \pi (1), \pi (n) \}$. As stated in the proof of Proposition \ref{prop2}, this set is minimum because a permutation graph is dominated by a single vertex $v$ when every number (in the one-line notation) to the right of $v$ is  less than $v$, and every number to the left of $v$ is greater than $v$. Every interior number of the permutation will either be less than or greater than both $\pi (1)$ and $\pi (n)$. Thus, $\{ \pi (1),\pi (n) \} $ forms a minimum dominating set of $G_\pi$.
\end{proof}
 
By computational evidence, the following algorithm often fails for permutations satisfying the two cases outlined by Propositions \ref{prop2} and \ref{prop3}. Excluding the aforementioned cases, this algorithm works for most permutations on $[n]$ for small $n$. As a point of reference, after excluding the two exceptional cases, this algorithm works for upwards of $96 \%$ of permutations on $[n]$ where $n \leq 8$. While this process does not guarantee a minimum dominating set, it is simple enough to easily do by hand. At the very least, this technique will obtain an upper bound of $\gamma (G_{\pi})$.

\begin{enumerate}
\item Find and list all of the decreasing subsequences of maximum length (equivalent to finding all maximum cliques).
\item For any number $k$ that is not a member of a maximum clique, find and list all maximal cliques that that contain $k$ (these cliques will not be maximum).
\item Find the most common element $c$ of all collected decreasing subsequences. If there are multiple most common elements, choose one. Add $c$ to the dominating set, and remove all listed subsequences containing $c$.
\item Repeat $3$ until the collected list of subsequences is empty.
\end{enumerate}

Once again, it is clear that this procedure will generate a dominating set, but not necessarily a minimum dominating set.

\section{A Consequence to Strong Fixed Points}

\begin{definition}
A \textit{strong fixed point} of a permutation $\pi$ on $[n]$ is an element $k$ such that $\pi^{-1}(j)<\pi^{-1}(k)$ for all $j<k$, and $\pi^{-1}(i)>\pi^{-1}(k)$ for all $i>k$.
\end{definition}
There is a bijection between the permutations on $[n]$ with exactly $k$ strong fixed points and the permutation graphs on $n$ vertices with exactly $k$ dominating sets of size one. The bijection is given by reversing a permutation. If $\pi=[a_1,a_2,\dots,a_n]$, then the reverse of $\pi$ is $\pi^r=[a_n,\dots,a_2,a_1]$. We can deduce from Lemma \ref{lem1} that $\{k\}$ is a dominating set of a permutation graph $G_\pi$ if and only if $k$ is a strong fixed point of $\pi^r$.

Thus, if we denote $St(n,k)$ as the number of permutations on $[n]$ with exactly $k$ strong fixed points ($A145878$ on the OEIS), then $St(n,k)=f(n,1,k)$, so Theorem \ref{recthm} applies to $St(n,k)$. Similarly, $g(n,1)$, whose values can be calculated as in Theorem \ref{dom1count}, is the number of permutations on $[n]$ with at least one strong fixed point ($A006932$ on the OEIS). We can use these formulas to give an inductive proof of some conjectures stated on the OEIS. A few of these conjectures are:
\begin{align*}
St(k+3,k)&=3(k+1),\\
St(k+4,k)&=(k+1)(k+28)/2,\\
St(k+5,k)&=(k+1)(3k+77).
\end{align*}
Notice that $St(k,k)=1$, since the only permutation on $[k]$ with exactly $k$ strong fixed points is $[1,2,\dots,k]$. Additionally, $St(k+1,k)=0$ since it is impossible for any graph on $k+1$ vertices to have exactly $k$ dominating sets of size one. The equality $St(k+2,k)=k+1$ can be verified by the recursive formulas of Theorem \ref{recthm} as follows:
\begin{align*}
St(2,0)=1,\\
\end{align*}
and for $k \geq 1$ and working inductively,
\begin{align*}
St(k+2,k)&=f(k+1,1,k-1)f(0,1,0)+f(k,1,k-1)f(1,1,0)+f(k-1,1,k-1)f(2,1,0)\\
&=k+1.
\end{align*}
We can prove the above conjectures in a similar fashion.
\begin{corollary}\label{cor1}
$$St(k+3,k)=3(k+1)$$
\end{corollary}
\begin{proof}
It is easily verifiable that $St(3,0)=3$ (notice that this value is also given by $3!-g(3,1)$). Now, working inductively, we have for $k \geq 1$:
\begin{align*}
St(k+3,k)&=f(k+2,1,k-1)f(0,1,0)+f(k+1,1,k-1)f(1,1,0)+f(k,1,k-1)f(2,1,0)+f(k-1)f(3,1,0)\\
&=3k+3=3(k+1).
\end{align*}
\end{proof}
The proofs of the following two corollaries are comparable to that of Corollary \ref{cor1}.
\begin{corollary}
For any natural number $k$, $St(k+4,k)=(k+1)(k+28)/2$
\end{corollary}

\begin{corollary}
For any natural number $k$, $St(k+5,k)=(k+1)(3k+77)$
\end{corollary}

We now give a recursive method of finding closed expressions for the values $St(k+r,k)$ when $r$ is fixed and $k$ is allowed to vary. More specifically, we show that $St(k+r,k)$ is given by a polynomial expression for a fixed $r$, and that if the polynomial expressions for $St(k+s,k)$ for $s<k$ are known, then the polynomial expression for $St(k+r,k)$ can be determined. 
\begin{theorem}\label{Stthm}
The value $St(k+r,k)$ for fixed $r$ is given by a polynomial expression $p(k)$. Furthermore, for $k \geq 1$, if $R(k)=\sum_{i=2}^{r+1}St((k+r)-i,k-1)St(i-1,0)=b_{n-1}k^{n-1}+ \cdots + b_1k+b_0$, then $p(k)$ is an $n^{\rm{th}}$ degree polynomial and the coefficients of $p$ are given by
\[
a_n=\frac{b_{n-1}}{n},
\]
and for $1 \leq j \leq n-1$,
\[
a_{n-j}=\frac{b_{n-j-1}-\sum\limits_{i=0}^{j-1}(-1)^{j-i} \binom{n-i}{j+1-i}a_{n-i}}{{n-j \choose 1}},
\]
and $a_0=St(r,0)$.
\end{theorem}
\begin{proof}
We prove this theorem by induction. Notice from what we have shown above that $St(k+t,k)$ is given by a polynomial expression for $0 \leq t \leq5$. Assume that $St(k+s,k)$ is given by a polynomial expression in $k$ for all $s<r$. Let $p(k)=a_nk^n+ \cdots +a_1k+a_0$, where the $a_i$'s are given as above. Notice that $St(r,0)=p(0)=a_0$. Assume that $St((k-1)+r,k-1)=p(k-1)$. Notice that for all $1 \leq j \leq n-1$, $\sum_{i=0}^{j}(-1)^{j-i} {n-i \choose j+1-i}a_{n-i}=b_{n-j-1}$. By Theorem \ref{recthm},
\begin{align*}
St(k+r,k)&=\sum_{i=1}^{r+1}St((k+r)-i,k-1)St(i-1,0)\\
&=p(k-1)+\sum_{i=2}^{r+1}St((k+r)-i,k-1)St(i-1,0)\\
&=p(k-1)+R(k)\\
&=a_n(k-1)^n+ \cdots +a_1(k-1)+a_0+R(k)\\
&=p(k)+\sum_{i=0}^{n-1}a_{n-i}\sum_{j=1}^{n-i}{n-i \choose j}(-1)^{j}k^{(n-i)-j}+R(k),
\end{align*}
where the last equality is obtained by using the binomial expansion of each term of the form $(k-1)^c$. Finally, by collecting terms with like powers of $k$ in the last expression, we have
\begin{align*}
St(k+r,k)&=p(k)+\sum_{j=0}^{n-1}k^{n-(j+1)}\sum_{i=0}^j(-1)^{j+1-i}{n-i \choose j+1-i}a_{n-i}+R(k)\\
&=p(k)+\sum_{j=0}^{n-1}(-b_{n-j-1})k^{n-j-1}+R(k)\\
&=p(k)-R(k)+R(k)\\
&=p(k),
\end{align*}
completing the inductive step.
\end{proof}
Recall that $St(m,0)=f(m,1,0)=m!-g(m,1)$, which can be calculated recursively by Theorem \ref{dom1count}. Furthermore, all of the values and expressions in Theorem \ref{Stthm} can be calculated recursively.
\section{Conclusion} 

Here we found a recursive formula for $g(n,1)$. We would like to obtain similar results for $g(n,k)$ for different values of $k$, however, this seems to be significantly more difficult than the case $k=1$. 

\begin{problem}
Let $k$ be a positive integer greater than $1$.  Does there exist a recursive formula for $g(n,k)$ for every $k$?
\end{problem} 

\begin{problem}
Let $k$ be a positive integer greater than $1$.  Does there exist a recursive formula for $g(n,k)$ for a fixed $k$?
\end{problem}

In a similar fashion, this study examined $f(n,k,t)$.  A natrual question is to study $f(n,k,t)$ for different values of $k$.

\begin{problem}
Let $k$ be a positive integer.  Can we find $f(n,k,t)$ for a fixed $t$?
\end{problem}

It is also of interest to find the number of permutation graphs with domination number $2$ that are dominated by a particular pair of vertices.

\begin{problem}
Can one count the number of permutation graphs with domination number $2$ that are dominated by the set $\{u,v\}$ for a fixed $u,v \in V(G)$?
\end{problem}

\section{Acknowledgements}

These results are based upon work supported by the National Science Foundation under the grant number DMS-1560019.  We would also like to send our gratitude to Eugene Fiorini and Byungchul Cha for making all of this possible.  


\end{document}